\documentclass{amsart}

\usepackage{amsfonts,amsmath,amssymb,amsthm}
\usepackage{mathrsfs,mathtools,stmaryrd,wasysym}
\usepackage{enumerate}
\usepackage{xspace} 

\usepackage{esint}
\usepackage{hyperref}
\usepackage{xcolor}
\DeclareMathAlphabet{\mathpzc}{OT1}{pzc}{m}{it}

\newcommand{\TheTitle}{Self-improving properties for a class of elliptic and parabolic equations on bounded domains}
\newcommand{\ShortTitle}{Self-improving properties for linear and semi-linear equations}

\theoremstyle{plain}
\newtheorem{theorem}{Theorem}
\newtheorem{proposition}[theorem]{Proposition}
\newtheorem{lemma}[theorem]{Lemma}
\newtheorem{corollary}[theorem]{Corollary}

\theoremstyle{definition}

\newtheorem{remark}[theorem]{Remark}

\newcommand{\R}{\mathbb{R}}
\newcommand{\polH}{\mathbb{H}}
\newcommand{\calL}{\mathcal{L}}
\newcommand{\calD}{\mathcal{D}}
\newcommand{\calB}{\mathcal{B}}
\newcommand{\bv}{\mathbf{v}}
\newcommand{\GRAD}{\nabla}
\DeclareMathOperator{\DIV}{div}
\newcommand{\diff}{\, \mbox{\rm d}}
\newcommand{\ie}{i.e.,\@\xspace}

\newcommand{\bF}{{\mathbf{F}}}
\newcommand{\bL}{{\mathbf{L}}}
\newcommand{\bW}{{\mathbf{W}}}

\newcommand{\bef}{{\mathbf{f}}}
\newcommand{\beps}{{\boldsymbol{\varepsilon}}}
\newcommand{\bw}{{\mathbf{w}}}
\newcommand{\bu}{{\mathbf{u}}}

\DeclareMathOperator{\supp}{supp}
\DeclareMathOperator*{\essinf}{essinf}

\newcommand{\calP}{{\mathcal{P}}}

\newcommand{\mae}{~a.e.~}
\newcommand{\vp}{{\mathrm{v.p.}}}

\begin{document}

\title[\ShortTitle]{\TheTitle}

\author{Tadele Mengesha}
\email[(T.~Mengesha)]{mengesha@utk.edu}

\author{Abner J.~Salgado}
\email[(A.J.~Salgado)]{asalgad1@utk.edu}

\address{Department of Mathematics, University of Tennessee, Knoxville, TN 37996, USA}

\subjclass[2020]{35D30,   
35D10,    
35R11,    
46B70.    
}

\keywords{fractional differentiability; nonlocal equations; Lipschitz domains; self-improvement; interpolation spaces.}

\begin{abstract}
  We discuss self improving properties of some local and nonlocal, elliptic and parabolic, equations on bounded domains. We employ a functional analytic approach wherein the solution space sits in a suitable interpolation scale. Utilizing a classical analytic perturbation result, we extrapolate the invertibility of the main operator from the base space to nearby spaces within the interpolation family.
\end{abstract}

\maketitle

\begin{center}
  \emph{We dedicate this work to the memory of our colleague and friend Tim Schulze.}
\end{center}

\section{Introduction}
\label{sec:intro}

The purpose of this work is to discuss the so-called self-improving properties of solutions to elliptic and parabolic, local and nonlocal, linear and semi-linear, problems on bounded domains. These properties are characterized by solutions exhibiting improved spatial gradient integrability for local elliptic and parabolic equations, \cite{MR159110,MR417568,Kinnunen2000}, as well as improved integrability and spatial fractional differentiability for nonlocal elliptic and parabolic equations, \cite{MR3336922,MR3907738,Mengesha-Scott-FKorn}. These hold under minimal assumptions on the problem data. Self-improving properties are also known for double phase problems \cite{MR4669309,MR4507377}, nonlocal porous medium equations \cite{MR4019094,MR4149740}, and for problems with nonstandard growth \cite{Bogelein2011, ERHARDT20161772}, to cite a few extensions.

In the literature, self-improvement properties for solutions to nonlocal equations typically assume that the problem is posed on the entire space $\R^d$, $d \geq 1$. Consequently, the resulting estimates involve either localized norms or global norms defined over all of $\R^d$. In applications, however, one frequently encounters problems formulated on bounded Lipschitz domains. This was the case in our recent work \cite{Mengesha2025} where having a global result, on a bounded Lipschitz domain, proved crucial. In the absence of citable literature, we presented a brief sketch in \cite[Appendix B]{Mengesha2025} detailing how such global result can be obtained. It is our goal here to expand on \cite[Appendix B]{Mengesha2025} to obtain self-improving properties for local and nonlocal, elliptic and parabolic, problems over bounded Lipschitz domains. We also present some results, mostly of perturbation-type, regarding semi-linear problems.

The standard approach to prove higher regularity results rely on the classical strategy used for second-order partial order differential equations, which utilizes Caccioppoli inequalities to derive reverse H\"older estimates, followed by a Gehring's lemma-type result to obtain improved integrability. This approach is robust and is implemented for various problems (linear, nonlinear, nonlocal) studied in most of the above literature cited above.  

In contrast, reference \cite{MR3907738} presents a functional analytic approach for linear nonlocal elliptic and parabolic equations that is different from the classical approach. The strategy, applied to nonlocal equations posed over $\R^d$, relies on the so-called \emph{Shneiberg's qualitative analytic perturbation lemma}, its extensions and variations. Our current contribution starts with the observation that one can adapt the strategy in  \cite{MR3907738} to problems that are posed over bounded Lipschitz domains. To that end, we begin by stating Shneiberg's result in a form that is suitable for our purposes.

\begin{lemma}[Shneiberg]
\label{lem:Shneiberg}
  Let $\bar{X} = (X_0,X_1)$ and $\bar{Y} = (Y_0,Y_1)$ be interpolation couples of Banach spaces and $T : X_j \to Y_j$ be a bounded linear operator ($j = 0,1$).
   \begin{enumerate}[1.]
     \item Complex interpolation: Let, for $\theta \in (0,1)$, denote by $\bar{X}_\theta \coloneqq [X_0,X_1]_\theta$ the scale of complex interpolation spaces between $\bar{X}$. The symbol $\bar{Y}_\theta$ carries a similar meaning. Assume that there is $\alpha \in (0,1)$ such that $T^{-1}$ exists and it is a bounded operator from $\bar{Y}_\alpha$ to $\bar{X}_\alpha$. Then, there is $\delta>0$ such that, if
     \[
       |\theta - \alpha |<\delta,
     \]
     then $T^{-1}$ exists and it is a bounded operator from $\bar{Y}_\theta$ to $\bar{X}_\theta$. The operator norm depends only on $\| T^{-1} \|_\alpha$ and $\delta$.

     \item Real interpolation: Let, for $\theta \in (0,1)$ and $p \in [1,\infty)$, denote by $\bar{X}_{\theta,p}$ the scale of real interpolation spaces between $\bar{X}$. $\bar{Y}_{\theta,p}$ is defined similarly. Assume that there are $\alpha \in (0,1)$ and $q \in [1,\infty)$ such that $T^{-1}$ exists and it is a bounded operator from $\bar{Y}_{\alpha,q}$ to $\bar{X}_{\alpha,q}$. Then, there is $\delta>0$ such that, if
     \[
       |\theta - \alpha |<\delta,
     \]
     then $T^{-1}$ exists and it is a bounded operator from $\bar{Y}_{\theta,q}$ to $\bar{X}_{\theta,q}$. The operator norm depends only on $\| T^{-1} \|_{\alpha,q}$ and $\delta$.
   \end{enumerate}
\end{lemma}
\begin{proof}
  See \cite{MR634681}, or \cite[Theorem A.1]{MR3907738} for the complex interpolation case; and \cite[Theorem 2.8]{MR569253} for the real interpolation case.
\end{proof} 
In the following sections, we select function spaces that will serve as solution spaces to the boundary value problems we study. We then characterize the spaces as interpolation spaces, and ensure the valid application of the above perturbation result.  Our presentation is organized as follows. Notation and preliminary results are discussed in Section~\ref{sec:Prelims}. In particular, we discuss the interpolation of Sobolev-Bochner spaces over Lipschitz domains. Nonlocal elliptic equations, both linear and semi-linear, and their self-improvement properties are discussed in Section~\ref{sec:EllipticEquations}. We next move to elliptic systems in Section~\ref{sec:EllipticSystems}. Finally, nonlocal parabolic equations are discussed in Section~\ref{sec:Parabolic}.

\section{Preliminaries}
\label{sec:Prelims}

Let us begin by stating some needed results. All our vector spaces are over the real numbers, and we shall adhere to standard notation with regards to function spaces. In what follows $\Omega \subset \R^d$, $d \geq 1$, is a bounded Lipschitz domain and, when dealing with time dependent problems $T>0$ is a positive final time. We define the space-time cylinders
\[
  Q_T \coloneqq (0,T) \times \Omega, \qquad Q_\R \coloneqq \R \times \Omega.
\]
As usual, for $p \in [1,\infty]$, we denote by $p'$ its H\"older conjugate.

We refer the reader to \cite{MR482275,MR3753604} for definitions and results about interpolation, real and complex, of Banach spaces. For example, for a given $\alpha\in (0, 1)$ and $1<p<\infty$, the fractional Sobolev space $W^{\alpha, p}(\mathbb{R}^{d})$ coincides with the Besov class $B^{\alpha}_{p\,p}(\mathbb{R}^{d})$ and consists of functions in $L^{p}(\mathbb{R}^{d})$ with finite Gagliardo semi-norm. Following the notation in \cite{MR1951822}, the function space $ \widetilde{W}^{\alpha,p}(\Omega)$ is defined as 
\[
  \widetilde{W}^{\alpha,p}(\Omega) \coloneqq \left\{ u \in W^{\alpha,p}(\mathbb{R}^{d}): \supp(u)\subset \overline{\Omega} \right\}. 
\]
In addition, we shall need some interpolation results regarding function spaces on Lipschitz domains. In what follows, we let $p_0, p_1 \in (1,\infty)$, $\alpha_0,\alpha_1 \in (0,1)$, and, for $\theta \in (0,1)$, we define
\[
  \frac1p = \frac{1-\theta}{p_0} + \frac{\theta}{p_1}, \qquad \alpha = (1-\theta)\alpha_0 + \theta \alpha_1.
\]

\begin{enumerate}[$\bullet$]
  \item Complex interpolation of Lebesgue spaces: For $\theta \in (0,1)$ we have
  \begin{equation}
  \label{eq:LpComplex}
    [L^{p_0}(\Omega), L^{p_1}(\Omega) ]_\theta = L^p(\Omega) .
  \end{equation}

  \item Complex interpolation of fractional Sobolev spaces (\cite{MR1951822}):
  \begin{equation}
  \label{eq:WapComplex}
    [\widetilde{W}^{\alpha_0,p_0}(\Omega), \widetilde{W}^{\alpha_1,p_1}(\Omega) ]_\theta = \widetilde{W}^{\alpha,p}(\Omega).
  \end{equation}
  
  \item Real interpolation of fractional Sobolev spaces (\cite{MR1951822}):
  \begin{equation}
  \label{eq:Wa2Real}
    ( \widetilde{W}^{\alpha_0,2}(\Omega), \widetilde{W}^{\alpha_1,2}(\Omega) )_{\theta,2} = \widetilde{W}^{\alpha,2}(\Omega).
  \end{equation}

  \item Real interpolation of Bochner spaces (\cite{MR358326}): Let $\bar{X} = (X_0,X_1)$ be an interpolation couple. Then
  \begin{equation}
  \label{eq:Cwikel}
    \left( L^{p_0}(0,T;X_0), L^{p_1}(0,T;X_1) \right)_{\theta,p} = L^p\left(0,T; (X_0,X_1)_{\theta,p} \right).
  \end{equation}
\end{enumerate}

\subsection{Sobolev-Bochner spaces on Lipschitz domains}
\label{sec:Stuff}

To deal with parabolic problems we need to characterize the real interpolation of Sobolev-Bochner spaces. We begin by establishing some notation. Let, for $\sigma \in \R$, $(-\Delta)^\sigma$ denote the $\sigma$-th root of the Dirichlet Laplacian on $\Omega$, in the sense of spectral theory. We define
\[
  \polH^\sigma(\Omega) \coloneqq \calD\left( (-\Delta)^{\sigma/2} \right).
\]
Note that $L^2(\Omega) = \polH^0(\Omega)$ and recall, see \cite[Theorem 4.36]{MR3753604}, that these spaces form a real interpolation scale, \ie
\[
  \polH^{\sigma}(\Omega) = \left( \polH^{\sigma_1}(\Omega), \polH^{\sigma_2}(\Omega)\right)_{\theta,2}, \qquad \sigma = (1-\theta)\sigma_1 + \theta \sigma_2.
\]

We now follow \cite{MR430814} and \cite[\S1.13.3]{MR503903} and let
\[
  A \coloneqq L^2(Q_\R) = L^2(\R;L^2(\Omega)).
\]
On this space we define two semigroups:
\begin{enumerate}[$\bullet$]
  \item The \emph{semigroup of translations in time} $\{G_1(\tau)\}_{\tau \geq 0}$. For $w \in L^2(Q_\R)$
  \[
    G_1(\tau) w(t,x) = w(t+\tau,x).
  \]
  Its infinitesimal generator is the time derivative $\Lambda_1 \coloneqq \diff_t$, and its domain is
  \[
    \calD(\Lambda_1) \coloneqq \left\{ w \in L^2(Q_\R) : \diff_tw \in L^2(Q_\R) \right\} = W^{1,2}(\R; L^2(\Omega)).
  \]

  \item Let $\sigma \in (0,1]$ and $\Lambda_2 \coloneqq (-\Delta)^\sigma$ with domain
  \[
    \calD(\Lambda_2) \coloneqq \left\{ w \in L^2(Q_\R) : (-\Delta)^\sigma w \in L^2(Q_\R) \right\} = L^2(\R; \polH^{2\sigma}(\Omega) ).
  \]
  This operator generates a \emph{heat-like} semigroup $\{G_2(h) = \exp(-h \Lambda_2) \}_{h \geq 0}$ defined as follows. For $w \in L^2(Q_\R)$ we let $W : (y,t,x) \in \R \times Q_\R \mapsto W(y,x;t) \in \R$ solve
  \[
    \partial_y W(y,x;t) + \Lambda_2 W(y,x;t) = 0, \qquad W(0,x;t) = w(t,x).
  \]
  Then
  \[
    G_2(h) w(t,x) = W(h,x;t).
  \]
  Notice that the time variable here acts as a parameter.
\end{enumerate}

We remark that, from their definition, it follows that these semigroups commute. Define
\[
  K \coloneqq \calD(\Lambda_1) \cap \calD(\Lambda_2),
  \qquad
  \| w \|_K \coloneqq \left( \| w \|_A^2 + \| \Lambda_1 w \|_A^2 + \| \Lambda_2 w \|_A^2 \right)^{1/2}.
\]
We let $\theta \in (0,1)$, $p \in (1,\infty)$, and define
\[
  \calB^\theta_p \coloneqq ( A, K )_{\theta,p}.
\]

The following results are taken from \cite{MR430814}.

\begin{theorem}[interpolation]
\label{thm:Schmeisser}
  In the setting we have devised above, we have:
  \begin{enumerate}[$\bullet$]
    \item \emph{Characterization:}
    \[
      \calB^\theta_p = ( A, \calD(\Lambda_1) )_{\theta,p} \cap ( A, \calD(\Lambda_2) )_{\theta,p}.
    \]

    \item \emph{Reiteration I:} If $\lambda \in (0,1)$ and $r \in (1,\infty)$, then
    \[
      (A, \calB^\theta_p )_{\lambda,r} = \calB^{\theta \lambda}_r.
    \]
    
    \item \emph{Reiteration II:} Let $\theta_1,\theta_2, \lambda \in (0,1)$, $p_1,p_2,r \in (1,\infty)$. Then
    \begin{equation}
    \label{eq:ReiterationII}
      \left( \calB^{\theta_1}_{p_1},\calB^{\theta_2}_{p_2} \right)_{\lambda,r} = \calB^\mu_r,
    \end{equation}
    where
    \[
      \mu = (1-\lambda)\theta_1 + \lambda\theta_2.
    \]
  \end{enumerate}
\end{theorem}
\begin{proof}
  All these results adapt \cite{MR430814} to our setting. More precisely, referring to \cite{MR430814}, the characterization is Theorem 1b, Reiteration I is (13) and Reiteration II is (14), respectively.
\end{proof}

We now focus on the function spaces $\calB^\theta_2$, as they are the ones that will be used in our study of parabolic problems later. To that end, let us give then an explicit description of these spaces. A related result can be found in \cite[Example 3.6.5]{MR3930629}, but under the assumption that $\Omega$ is a \emph{corner}.

\begin{corollary}[Sobolev-Bochner spaces]
  In the setting of Theorem~\ref{thm:Schmeisser}, we have
  \begin{equation}
  \label{eq:Bt2characterize}
    \calB^\theta_2 = W^{\theta,2}(\R;L^2(\Omega)) \cap L^2(\R; \polH^{2\sigma\theta}(\Omega)).
  \end{equation}
  The family $\{\calB^\theta_2\}_{\theta \in (0,1)}$ forms a real interpolation scale. Moreover, if $\theta_1,\theta_2,\lambda \in(0,1)$, then 
\begin{equation*}
 \label{eq:Bt2reiterate}
  \begin{aligned}
    &\left( \calB^{\theta_1}_2, \calB^{\theta_2}_2 \right)_{\lambda,2} \\
    &= \left(
      W^{{\theta_1},2}\left( \R; L^2(\Omega) \right)
        \cap
      L^2\left( \R; \polH^{2\sigma\theta_1}(\Omega) \right)
      ,
      W^{{\theta_2},2}\left( \R; L^2(\Omega) \right)
        \cap
      L^2\left( \R; \polH^{2\sigma\theta_2}(\Omega) \right)
    \right)_{\lambda,2} \\
    &= W^{\mu,2}\left( \R; L^2(\Omega) \right)
      \cap
      L^2\left(\R; \polH^{2\sigma\mu}(\Omega) \right), 
  \end{aligned}
\end{equation*}
    where
  \[
    \mu = (1-\lambda)\theta_1 + \lambda\theta_2.
  \]
\end{corollary}
\begin{proof}
  Let us first characterize the interpolation between $A$ and the domain of each generator. First we use \cite[Theorem 2.7.4 and (3.6.3)]{MR3930629} to obtain
  \[
    (A,\calD(\Lambda_1) )_{\theta,2} =
    \left(
      L^2\left( \R; L^2(\Omega) \right)
    ,
      W^{1,2}\left( \R; L^2(\Omega) \right)
    \right)_{\theta,2}
    = W^{\theta,2}\left( \R; L^2(\Omega) \right).
  \]
  Next, we apply \eqref{eq:Cwikel} to obtain
  \[
    (A,\calD(\Lambda_2) )_{\theta,2} =
    \left(
      L^2\left( \R; L^2(\Omega) \right)
    ,
      L^2\left( \R; \polH^{2\sigma}(\Omega) \right)
    \right)_{\theta,2} =
    L^2(\R;\polH^{2\sigma\theta}(\Omega)).
  \]
  Notice that this is the one and only step that requires us to restrict the secondary interpolation index to $q=2$. Finally, we apply these in the characterization given in Theorem~\ref{thm:Schmeisser} to obtain \eqref{eq:Bt2characterize}.
  
  The reiteration  is a restatement of \eqref{eq:ReiterationII}, together with \eqref{eq:Bt2characterize}.
\end{proof}

\section{Nonlocal elliptic equations}
\label{sec:EllipticEquations}

We now begin our description of the self-improving properties \emph{per se} and extend the higher integrability and differentiability result of \cite[Appendix B]{Mengesha2025}. 

\subsection{Linear problems}
We let $A \in L^\infty(\R^d \times \R^d)$ be symmetric, \ie $A(x,y) = A(y,x)$ for almost every $(x,y) \in \R^{2d}$ and strictly positive. Set
\begin{equation}\label{Ellipticity-Boundedness-for-nonlocal}
  \lambda = \essinf_{(x,y) \in \R^{2d}} A(x,y) >0, \qquad \qquad \Lambda = \| A \|_{L^\infty(\R^{2d})}.
\end{equation}
Fix $s \in (0,1)$ and consider the following problem. Find $u : \R^d \to \R$ such that
\begin{equation}
\label{eq:NonLocalEllipticScalar}
  \begin{dcases}
    \vp \int_{\R^d} A(x,y) \frac{u(x) - u(y) }{|x-y|^{d+2s} } \diff y = F(x), & x \in \Omega, \\
    u = 0, &\text{ in } \R^d \setminus \Omega.
  \end{dcases}
\end{equation}

\begin{theorem}[nonlocal elliptic equations]
\label{thm:OurResult}
  There are $\calP >2$ and $\epsilon_0>0$ that depend on $d$, $\Omega$, and $\Lambda/\lambda$ such that, for every $p \in [\calP',\calP]$, there is a corresponding $\sigma\in [s-\epsilon_0, s+\epsilon_0]$ such that if $\rho = 2s - \sigma$ and $F \in W^{-\rho,p}(\Omega)$, then \eqref{eq:NonLocalEllipticScalar} has a unique solution $u \in \widetilde{W}^{\sigma,p}(\Omega)$, with the corresponding estimate. Moreover, if $p\in (2, \calP]$, then $\sigma>s$.
\end{theorem}
\begin{proof}
  We recall that, for Lipschitz domains, the scale of spaces $\widetilde{W}^{\alpha,q}(\Omega) = \bar{B}^\alpha_{qq}(\Omega)$ was studied in \cite[Theorem 3.5]{MR1951822}, where its complex and real interpolation properties are discussed; see \eqref{eq:WapComplex} and \eqref{eq:Wa2Real}. Observe that, for $\alpha>0$, it is customary to denote the dual of $\widetilde{W}^{\alpha,q'}(\Omega)$ by $W^{-\alpha,q}(\Omega)$.

  Let $p \in (1,\infty)$ and $\theta, \nu \in (0,1)$ with $\theta + \nu = 2s$. We let the operator $\calL_A : \widetilde{W}^{\theta,p}(\Omega) \to W^{-\nu,p}(\Omega)$ be defined as
  \[
    \langle \calL_A w, \varphi \rangle \coloneqq \int_{\R^d} \int_{\R^d} A(x,y) \frac{w(x) - w(y) }{|x-y|^{\theta}} \frac{\varphi(x) - \varphi(y) }{|x-y|^{\nu}} \frac{\diff x \diff y}{|x-y|^d},
  \]
  where we used that $\theta + \nu = 2s$. This is a bounded linear operator. Indeed,
  \begin{multline*}
    \left| \langle \calL_A w, \varphi \rangle\right| \leq \\
     \Lambda \left[\iint_{\R^{2d}} \left| \frac{w(x) - w(y) }{|x-y|^{\theta}} \right|^p \frac{\diff x \diff y}{|x-y|^d} \right]^{1/p}
    \left[ \iint_{\R^{2d}} \left| \frac{\varphi(x) - \varphi(y) }{|x-y|^{\nu}} \right|^{p'} \frac{\diff x \diff y}{|x-y|^d} \right]^{1/p'}
    \\
    \leq \Lambda |w|_{\widetilde{W}^{\theta,p}(\Omega)} |\varphi|_{\widetilde{W}^{\nu,p'}(\Omega)}.
  \end{multline*}
  Notice that here, implicitly, we are doing computations first for smooth and compactly supported functions, and then using their density. By fixing $\theta =s$ and $p=2$ we also get that
  \[
    \langle \calL_A w,w \rangle \geq \lambda \iint_{\R^{2d}} \left| \frac{w(x) - w(y) }{|x-y|^s} \right|^2 \frac{\diff x \diff y}{|x-y|^d} = \lambda |w|_{\widetilde{W}^{s,2}(\Omega)}^2,
  \]
  and so, by Lax-Milgram, this mapping is invertible and with a bounded inverse.

  Fix $\epsilon \in (0, \min\{s, 1-s\})$ and $\calP>2$. For $\theta \in (0,1)$ consider the interpolation scales
  \[
    \bar{X}_\theta = \left[ \widetilde{W}^{s-\epsilon,\calP'}(\Omega), \widetilde{W}^{s+\epsilon,\calP }(\Omega) \right]_\theta,
    \qquad \qquad
    \bar{Y}_\theta = \left[ W^{-s-\epsilon,\calP'}(\Omega), W^{-s+\epsilon,\calP}(\Omega) \right]_\theta.
  \]
  The previous considerations show that, for $\theta = \tfrac12$, the operator
  \[
    \calL_A^{-1} : \bar{Y}_{1/2} = W^{-s,2}(\Omega) \to \bar{X}_{1/2} = \widetilde{W}^{s,2}(\Omega)
  \]
  exists and it is bounded. Invoking then Lemma~\ref{lem:Shneiberg} we can assert that there is $\delta>0$ such that, for $|\theta - \tfrac12| < \delta$,
  \[
    \calL_A^{-1} : \bar{Y}_\theta = W^{-\rho, p}(\Omega) \to \bar{X}_\theta = \widetilde{W}^{\sigma,p}(\Omega),
  \]
  exists and it is bounded. Here we have set
  \[
    \sigma = s+\epsilon(2\theta-1), \qquad \frac1p = \frac{1-\theta}{\calP'}+ \frac\theta\calP, \qquad \rho = 2s-\sigma.
  \]
  In particular, if $\tfrac12 < \theta < \tfrac12 + \delta$, the claimed result follows.
\end{proof}

The self-improving property obtained in \cite[Theorem 31]{Mengesha2025}  can now be be stated as a particular case of the above as follows.

\begin{corollary}[self-improvement]
\label{self-im-L2data}
  Assume that $F \in L^2(\Omega)$. Then, there is $\epsilon > 0$ such that the solution to \eqref{eq:NonLocalEllipticScalar} satisfies $u \in \widetilde{W}^{s+\epsilon,2+\epsilon}(\Omega)$ with a corresponding estimate.
\end{corollary}
\begin{proof}
  It suffices to notice that $L^2(\Omega) \hookrightarrow W^{-s-\epsilon,2+\epsilon}(\Omega)$ for
  \[
    \epsilon \left( 1- \frac{d}4 \right) \geq -s,
  \]
  and invoke the previous result.
\end{proof}

\subsection{Semi-linear problems}

One can use the self-improvement property of linear equations to obtain a similar property for semi-linear nonlocal equations of the type  
\begin{equation}
\label{eq:nonlocal_bvp}
    \begin{dcases}
        \vp \int_{\mathbb{R}^d} A(x,y) \frac{u(x) - u(y)}{|x-y|^{d+2s}} \diff y = F(x,u), & x \in \Omega,
        \\
        u = 0, & \text{in } \mathbb{R}^d \setminus \Omega. 
    \end{dcases}
\end{equation} 
In fact, if we fix $p\geq 2$ and $\sigma \geq s$ satisfying the conclusion of Theorem \ref{thm:OurResult} then, for a class of nonlinearities $F$, problem \eqref{eq:nonlocal_bvp} has a solution $u\in \widetilde{W}^{s,2}(\Omega)$. In this case, the function $\widetilde{f}$, defined as,
\[
  \widetilde{f}(x) =  F(x, u(x)),
\]
belongs to $W^{-(2s-\sigma), p}(\Omega)$. Then $u$ solves the linear equation \eqref{eq:NonLocalEllipticScalar} with right hand side $\widetilde{f}$ and we may apply Theorem \ref{thm:OurResult}  to conclude that, in turn, $u$ belongs to $ \widetilde{W}^{\sigma,p}(\Omega)$. The following result provides a class of semi-linear nonlocal problems that fits this setting. 

\begin{proposition}[existence and dual fractional regularity]
\label{Existence-semilinear}
Suppose that $s\in (0, 1)$ and $d>2s$. 
Assume $F: \Omega \times \mathbb{R} \to \mathbb{R}$ is a Carath\'{e}odory function satisfying:
\begin{itemize}
  \item[(H1)](Growth condition) There exists a constant $C > 0$ such that for \mae $x \in \Omega$ and all $z \in \mathbb{R}$,
  \[
    |F(x,z)| \leq C(1 + |z|^q),
  \]
  where $1 < q < \frac{d+2s}{d-2s}$.
  
  \item[(H2)](Ambrosetti-Rabinowitz condition) There exist $\theta > 2$ and $M > 0$ such that for $|z| \geq M$,
  \[
    0 < \theta \mathcal{F}(x,z) \leq z F(x,z),
  \]
  where
  \[
    \mathcal{F}(x,z) = \int_0^z F(x,t) \diff t.
  \]
  
  \item[(H3)] (Local behavior) We have, uniformly in $x \in \Omega$,
  \[
    \lim_{|z| \to 0} \frac{F(x,z)}{|z|} = 0.
  \]
\end{itemize}
Then, problem \eqref{eq:nonlocal_bvp} has a non-trivial weak solution $\widetilde{u} \in \widetilde{W}^{s, 2}(\Omega)$. 
Moreover, if  $p \geq 2$ and $\sigma \in [s, 2s)$ are given with 
\[
  \sigma < d(1 - 1/p)
\]
and
\[
  q \leq \frac{2d}{d-2s} \left( \frac{1}{p} + \frac{2s-\sigma}{d} \right),
\]
in addition to satisfying {\textup{(H1)}}, then the nonlinear term satisfies the dual regularity $F(\cdot, \widetilde{u}(\cdot)) \in W^{-(2s-\sigma), p}(\Omega)$. Consequently, we have
\[
  \widetilde{u} \in \widetilde{W}^{\sigma,p}(\Omega).
\]
\end{proposition}
\begin{proof}
The existence of a nontrivial solution $\widetilde{u} \in \widetilde{W}^{s, 2}(\Omega)$ via the Mountain Pass Theorem is proved in \cite{Servadei2012}.  To prove that $\widetilde{f}$, defined as
\[
  \widetilde{f}(x) = F(x, \widetilde{u}(x)),
\]
belongs to $W^{-(2s-\sigma), p}(\Omega) = \left( \widetilde{W}^{(2s-\sigma), p'}(\Omega) \right)'$, where $1/p + 1/p' = 1$, we first note that the fractional Sobolev embedding theorem \cite{DiNezza2012} guarantees that the continuous embedding $\widetilde{W}^{(2s-\sigma), p'}(\Omega) \hookrightarrow L^{r'}(\Omega)$ holds when $\frac{1}{r'} = \frac{1}{p'} - \frac{(2s-\sigma)}{d}$. Passing to the conjugate exponents $r$ and $p$ we have
\[
  1 - \frac{1}{r} = 1 - \frac{1}{p} - \frac{(2s-\sigma)}{d}
  \qquad
  \iff 
  \qquad
  \frac{1}{r} = \frac{1}{p} + \frac{(2s-\sigma)}{d}.
\]
By duality, if $\widetilde{f} \in L^r(\Omega)$, then $\widetilde{f} \in W^{-(2s-\sigma), p}(\Omega)$. We must then show that $\widetilde{f} \in L^r(\Omega)$. To that end, using the growth bound from {(H1)}, we evaluate the $L^r$ norm of $\widetilde{f}$,
\[
  \int_\Omega \left| \widetilde{f} \right|^r \diff x = \int_\Omega \left| F(x, \widetilde{u}(x)) \right|^r \diff x 
  \leq C' \int_\Omega (1 + \left| \widetilde{u} \right|^{qr}) \diff x. 
\]
This integral is finite if $\widetilde{u} \in L^{qr}(\Omega)$. Since $\widetilde{u} \in \widetilde{W}^{s,2}(\Omega)$, we know $\widetilde{u} \in L^{2_s^*}(\Omega)$, where
\[
  2_s^* \coloneqq \frac{2d}{d-2s}.
\]
Thus, we require $qr \leq \frac{2d}{d-2s}$. Substituting our expression for $r$ yields
\[
  q \left( \frac{1}{1/p + (2s-\sigma)/d} \right) \leq \frac{2d}{d-2s} 
  \qquad
  \iff 
  \qquad
  q \leq \frac{2d}{d-2s} \left( \frac{1}{p} + \frac{(2s-\sigma)}{d} \right).
\]
But this  is exactly the additional condition imposed on $q$. Therefore, we have that $\widetilde{f} \in W^{-(2s-\sigma), p}(\Omega)$. 
Finally, we apply Theorem~\ref{thm:OurResult} to obtain that $\widetilde{u} \in \widetilde{W}^{\sigma,p}(\Omega)$.

This concludes the proof. 
\end{proof}

\begin{remark}[self-improvement]
From the proof of the result above, we observe that for nonlinearities with moderate growth in $z$, namely those with $1 < q < \frac{d}{d-2s}$, we can establish that $F(\cdot, \widetilde{u}) \in L^{2}(\Omega)$. Corollary \ref{self-im-L2data} then implies the existence of $\epsilon>0$ such that the solution $\widetilde{u}$ to \eqref{eq:nonlocal_bvp} belongs to $\widetilde{W}^{s+\epsilon, 2 + \epsilon}(\Omega)$.
\end{remark}

\section{Elliptic systems}
\label{sec:EllipticSystems}

We next consider elliptic systems, both local and nonlocal.

\subsection{Linearized elasticity with variable moduli}
We consider the following problem. Find $\bu : \bar\Omega \to \R^d$ such that
\begin{equation}
\label{eq:Elastic}
  \begin{dcases}
  -2\DIV( \mu \beps(\bu) ) - \GRAD( \lambda \DIV \bu ) = \bF, &\text{ in } \Omega,
  \\
  \bu = \boldsymbol{0}, & \text{ on } \partial \Omega,
  \end{dcases}
\end{equation}
where  $ \beps(\bu) \coloneqq \frac{1}{2}\left(\nabla \bu^\top + \nabla \bu\right)$ is the symmetric gradient of $\bu$ and  the Lam\'e parameters satisfy
\begin{equation}
\label{eq:lame-conditions}
 \lambda, \mu \in L^\infty(\Omega) \,\, \text{and} \, \, \lambda(x), \mu(x) \in [a,A], \quad \mae x \in \Omega,
\end{equation}
for some $0<a\leq A$.

Our goal now is to prove that solutions to \eqref{eq:Elastic} possess a higher integrability property akin to the classical one by Meyers, see \cite{MR159110}, for elliptic equations. We comment that results of this kind are plentiful in the literature.  For strongly elliptic systems satisfying the Legendre-Hadamard condition, which include \eqref{eq:Elastic} as a special case, this is shown in \cite[Theorem 2]{MR417568}; see also \cite{MR1282226} for such systems with mixed boundary conditions, and  \cite{MR1607245}  for a nonlinear version of the Stokes problem. In particular, in \cite[Appendix A]{Mengesha2025}, we have shown this. Here we will demonstrate that Lemma~\ref{lem:Shneiberg} can be used to obtain a proof.

\begin{proposition}[elliptic systems]
\label{prop:Meyers-elasticity}
  There is $\calP >2$ that depends on $d$, $\Omega$, and $\tfrac a A$, such that if $p \in [\calP',\calP]$ and $\bF  \in \bW^{-1,p}(\Omega)$ then problem \eqref{eq:Elastic} has a unique solution $\bu \in \bW^{1,p}_0(\Omega)$ with
  \[
    \| \GRAD \bu \|_{\bL^p(\Omega)} \leq C \| \bF  \|_{\bW^{-1,p}(\Omega)}.
  \]
  In particular, if $\bF \in \bL^2(\Omega)$, then there is $\bar{q}>2$ such that $u \in \bW^{1,\bar{q}}_0(\Omega)$ with the corresponding estimate.
\end{proposition}
\begin{proof}
Define the operator $\mathbb{L}$ via
\[
  \langle \mathbb{L}\bw, \bv\rangle \coloneqq \int_{\Omega}\left(2\mu(x)\beps(\bw(x)):\beps({\bv}(x)) + \lambda(x)\DIV\bw(x) \DIV \bv(x)\right) \diff x,
\]
so that the system of equations \eqref{eq:Elastic} may be rewritten as $\mathbb{L} \bu = \bF $.

Owing to the upper bound in \eqref{eq:lame-conditions}, for any $p \in (1,\infty)$, we have that $\mathbb{L}: \bW^{1,p}_0(\Omega) \to  \bW^{-1,p}(\Omega)$ is bounded. In addition, owing to Korn's inequality we have, for every $\bw\in \bW^{1,2}_0(\Omega)$
\[
  \langle \mathbb{L}\bw, \bw\rangle \geq a \| \GRAD \bw\|_{\bL^2(\Omega)}^2.
\]
As a consequence, we  may apply Lax-Milgram to conclude that $\mathbb{L}^{-1}$ exists and is bounded.  The remaining part of the proof is exactly the same as the proof of Theorem~\ref{thm:OurResult}.
\end{proof}

\subsection{Strongly coupled nonlocal systems}

Here for a given $s\in (0, 1)$ we consider the strongly coupled linear system of nonlocal equations given by 
\begin{equation}
\label{eq:SC-nonlocalsystem}
  \begin{dcases}
    \vp \int_{\mathbb{R}^{d}} \frac{ (x-y) \otimes (x-y) }{|x-y|^2} \left( \bu(x) - \bu(y) \right) \frac{A(x,y) \diff y}{|x-y|^{d+2s}} = \bF(x), & x \in \Omega,
    \\
    \bu = \boldsymbol{0}, &\text{in } \mathbb{R}^d \setminus \Omega,
  \end{dcases}
\end{equation}
where the coefficient $A \in L^{\infty}(\mathbb{R}^{2d})$ is symmetric and satisfies the ellipticity and boundedness assumptions given in \eqref{Ellipticity-Boundedness-for-nonlocal}.  This system of equations appears in applications such as in the so-called bond-based model of perydinamics.

Define the mapping $\mathfrak{L}_A$ via
\begin{multline*}
    \langle \mathfrak{L}_A\bv, \bw \rangle \coloneqq \\
     \iint_{\R^{2d}} A(x,y)
      \left( \bw(x) - \bw(y) \right)^\top \frac{ (x-y) \otimes (x-y) }{|x-y|^2} \left( \bv(x) - \bv(y) \right) \frac{\diff x \diff y}{|x-y|^{d+2s}},
\end{multline*}
and, just as in the case of nonlocal equations, for $p \in (1,\infty)$ and $\theta, \nu \in (0,1)$ with $\theta + \nu = 2s$, we have that, for any $\bv\in \widetilde{\mathbf{W}}^{\theta, p}(\Omega)$ and $\bw\in \widetilde{\mathbf{W}}^{\nu, p'}(\Omega)$,
\begin{align*}
  \left| \langle \mathfrak{L}_A\bv, \bw \rangle \right| &\leq \Lambda
  \left(
    \iint_{\R^{2d}}
      \left| \left( \bv(x) - \bv(y) \right)\cdot\frac{(x-y)}{|x-y|}\right|^{p}\frac{\diff x \diff y}{|x-y|^{d+p\theta}}
  \right)^{1/p}
  \\
  &\times
  \left(
    \iint_{\R^{2d}}
      \left| \left( \bw(x) - \bw(y) \right)\cdot\frac{(x-y)}{|x-y|}\right|^{p'}\frac{\diff x \diff y}{|x-y|^{d+p'\nu}}
  \right)^{1/p'}
  \\
  &\leq \Lambda |\bv|_{\widetilde{\mathbf{W}}^{\theta, p}(\Omega)}  |\bw |_{\widetilde{\mathbf{W}}^{\nu, p'}(\Omega)},
\end{align*}
and so $\mathfrak{L}_A : \widetilde{\mathbf{W}}^{\theta, p}(\Omega) \to \mathbf{W}^{-\nu, p}(\Omega)$ boundedly.
Moreover, for $p=2$ and $\theta=\nu=s$, we have
\begin{align*}
  \langle \mathfrak{L}_A\bv, \bv \rangle \geq
  \lambda \iint_{\R^{2d}}
    \left|\left( \bv(x) - \bv(y) \right)\cdot\frac{(x-y)}{|x-y|}\right|^{2}\frac{\diff x \diff y}{|x-y|^{d+2s}}
  \geq C \lambda|\bv|_{\widetilde{\mathbf{W}}^{s, 2}(\Omega)}^2,
\end{align*}
where the second inequality follows from the fractional Korn's inequality \cite[Theorem 1.1]{Mengesha-Scott-FKorn}. Again, by Lax-Milgram,
$\mathfrak{L}_A : \widetilde{\bW}^{s,2}(\Omega) \to \bW^{-s,2}(\Omega)$ is invertible. Proceeding as in the proof of Theorem~\ref{thm:OurResult}, we can establish the following result.

\begin{theorem}[nonlocal elliptic systems]
\label{thm:Ourresult-system}
  There are $\calP >2$ and $\epsilon_0>0$ that depend on $d$, $\Omega$, and $\Lambda/\lambda$ such that if $p \in [\calP',\calP]$ there is a corresponding $\sigma\in [s-\epsilon_0, s+\epsilon_0]$ such that if $\rho = 2s - \sigma$ and
  $\mathbf{F} \in \mathbf{W}^{-\rho,p}(\Omega)$, then \eqref{eq:SC-nonlocalsystem} has a unique solution $u \in \widetilde{\mathbf{W}}^{\sigma,p}(\Omega)$, with the corresponding estimate. Moreover, if $p\in (2, \calP]$, then $\sigma>s$.
\end{theorem}

We comment that this result improves upon \cite[Theorem 1.1]{Mengesha-Scott-FKorn} where, for a related problem, self-improvement in the interior was proved.

\begin{remark}[semi-linear systems]
As we have shown earlier, the self-improvement property of linear elliptic systems can be used to obtain a similar for semi-linear elliptic (local and nonlocal) systems. For example, consider the semi-linear system 
\begin{equation}
\label{eq:Elastic-semilinear}
\begin{dcases}
  -2\DIV( \mu \beps(\bu) ) - \GRAD( \lambda \DIV \bu ) = \bF  (\bu), &\text{ in } \Omega,
  \\
   \bu = \boldsymbol{0}, & \text{ on } \partial \Omega. 
  \end{dcases}
\end{equation}
Assume that $p\geq 2$ satisfies the conclusion of  Theorem \ref{prop:Meyers-elasticity}. Then for a class of nonlinearities $\bF$ a solution $\bu \in \bW^{1,2}_0(\Omega)$ to the semi-linear system  \eqref{eq:Elastic-semilinear} exists and the function
$\widetilde{\bef}$, defined as
\[
  \widetilde{\bef}(x) =  \bF(\bu(x)),
\]
belongs to $\bW^{-1, p}(\Omega)$. It then follows that $\bu$ solves the linear system \eqref{eq:Elastic} with right hand side $\widetilde{\bef}$ and we may apply Theorem \ref{prop:Meyers-elasticity}  to conclude that $ \bu$ is uniquely determined and belongs to $ \bW^{1,p}_0(\Omega)$.  Similar results can be stated for nonlocal semi-linear elliptic systems. 
\end{remark}

\section{Nonlocal parabolic equations}
\label{sec:Parabolic}
In this section, we give a description of the self-improving properties of nonlocal parabolic equation with the spatial variable in a bounded set.  See \cite{MR3907738} for a similar result for problems posed on $\mathbb{R}^{d}$, which motivated our work. 

\subsection{Linear problems}
Let us first describe the linear nonlocal parabolic problems we shall be interested in this section. We recall that, for any $s \in (0,1)$, $\widetilde{W}^{s,2}(\Omega) \hookrightarrow L^2(\Omega) \hookrightarrow W^{-s,2}(\Omega)$ forms a Gelfand triple. Since the mapping $\calL_A$, defined in the proof of Theorem~\ref{thm:OurResult}, is coercive on $\widetilde{W}^{s,2}(\Omega)$, for every $u_0 \in L^2(\Omega)$ and $F \in L^2(0,T;W^{-s,2}(\Omega))$, the problem
\begin{equation}
\label{eq:NLParabolic}
  \begin{dcases}
    \diff_t u + \calL_A u = F, & t \in (0,T),
    \\
    u(0)= u_0,
  \end{dcases}
\end{equation}
has a unique solution
\[
  u \in L^2(0,T; \widetilde{W}^{s,2}(\Omega) ) \cap W^{1,2}(0,T; W^{-s,2}(\Omega)) .
\]

Our final goal is to show self-improving properties for this problem. We must comment that, since our approach is based on a perturbation-type result, our statements are far from the so-called \emph{maximal regularity} estimates for evolution equations; see \cite[Chapter III]{MR1345385},  which, nevertheless, require different smoothness assumptions.

We begin with a well-known result that we shall improve upon.

\begin{proposition}[smoother data]\label{prop:smoothdata}
  Assume that, in  \eqref{eq:NLParabolic}, we have $u_0 \in \widetilde{W}^{s,2}(\Omega)$ and $F \in L^2(0,T;L^2(\Omega))$. Then
  \[
    u \in L^\infty(0,T;\widetilde{W}^{s,2}(\Omega)) \cap W^{1,2}(0,T;L^2(\Omega)).
  \]
\end{proposition}
\begin{proof}
  See, for instance, \cite[Chapter XVIII, Section 3.4, Remark 2]{MR1156075}.
\end{proof}

We begin with a characterization of the domain of $(-\Delta)^\sigma$.

\begin{lemma}[fractional Laplacian]
\label{lem:FractLaplace}
  Let $\Omega \subset \R^d$ be convex or with $C^2$ boundary. Then we have
  \[
    \polH^\sigma(\Omega) = \begin{dcases}
                             \widetilde{W}^{\sigma,2}(\Omega), & \sigma \in (0,1),
                             \\
                             W^{1,2}_0(\Omega) \cap \widetilde{W}^{\sigma,2}(\Omega), & \sigma \in [1,2],
                           \end{dcases}
  \]
  with equivalent norms.
\end{lemma}
\begin{proof}
  See \cite{MR216336} for the case of smooth domains and, for instance, \cite[Proposition 4.1.]{MR3356020} for convex domains.
\end{proof}

Here and in what follows we set $\sigma = s$ in the definition of the spaces $\calB^\theta_2$ that were introduced in Section~\ref{sec:Stuff}. Thus, we may recast \eqref{eq:NLParabolic} as: given $u_0 \in \polH^0(\Omega)$ and $F \in L^2(0,T; \polH^{-s}(\Omega))$, problem \eqref{eq:NLParabolic} has a unique solution
\[
  u \in L^2(0,T; \polH^s(\Omega) ) \cap W^{1,2}(0,T;\polH^{-s}(\Omega)).
\]

\begin{theorem}[nonlocal parabolic problems]
\label{thm:NLParabolic}
  Assume that $\Omega$ is convex or that it has a $C^2$ boundary; and that, in \eqref{eq:NLParabolic}, we have $u_0=0$ and $F \in L^2(0,T;L^2(\Omega))$. Then, for some $\delta >0$, the solution to \eqref{eq:NLParabolic} satisfies
  \begin{equation}
  \label{eq:HigherIntNLParabolic}
    u \in L^2( 0,T; \widetilde{W}^{s+\delta,2}(\Omega)) \cap W^{1/2+\delta}(0,T;L^2(\Omega)) .
  \end{equation}
  In addition, there is $\nu>0$, for which
  \begin{equation}
  \label{eq:BetterTimeDerivNLParabolic}
    u \in W^{1,2}(0,T; W^{-s+\nu,2}(\Omega)).
  \end{equation}
\end{theorem}
\begin{proof}
  We draw inspiration from \cite[Section 6]{MR3907738} and extend the problem, via the Fourier transform, to all of $\R$. This is why we require the initial condition to be zero. Once this extension is carried out, as originally described in \cite{MR200593}, the problem is well-posed in
  \begin{align*}
    \calB_2^{1/2} &= W^{1/2,2}(\R;L^2(\Omega)) \cap L^2(\R; \polH^{s}(\Omega))
    \\
      &= W^{1/2,2}(\R;L^2(\Omega)) \cap L^2(\R; \widetilde{W}^{s,2}(\Omega)),
  \end{align*}
  where we use the characterization of Lemma~\ref{lem:FractLaplace}.
  
  Over the family $\{\calB_2^\lambda\}_{\lambda \in (0,1)}$ define the operator
  \[
    \wp : \calB^\lambda_2 \to \left( \calB^\lambda_2 \right)', \qquad w \mapsto \diff_t w + \calL_A w.
  \]
  As in \cite[Lemma 6.7]{MR3907738} we see that this operator is well-defined, linear, and bounded for all $\lambda \in (0,1)$. Thus, the previous considerations and Lemma~\ref{lem:Shneiberg} imply that there is $\delta >0$ such that if
  \[
    \left|\mu - \frac12 \right| \leq \delta
  \]
  then $\wp^{-1}: (\calB_2^\mu)' \to \calB_2^\mu$ boundedly. Thus, the parabolic problem is well-posed in
  \begin{align*}
    \calB_2^{1/2+\delta} &= W^{1/2+\delta,2}(\R; L^2(\Omega)) \cap L^2( \R ; \polH^{s+\delta}(\Omega))
    \\
      &= W^{1/2+\delta,2}(\R; L^2(\Omega)) \cap L^2( \R ; \widetilde{W}^{s+\delta,2}(\Omega)),
  \end{align*}
  where we, again, used Lemma~\ref{lem:FractLaplace}. Now, a simple zero extension shows that
  \[
    L^2(0,T;L^2(\Omega)) \subset L^2(\R;L^2(\Omega)) \subset (\calB_2^{1/2+\delta})'.
  \]
  Thus, if $F \in L^2(0,T;L^2(\Omega))$, the solution to \eqref{eq:NLParabolic}, by restriction, satisfies \eqref{eq:HigherIntNLParabolic}.

  Finally, using the equation we have, for $\varphi \in C_0^\infty(Q_T)$ and $\nu \in (0,\delta]$,
  \begin{align*}
    \left| \int_0^T \langle u, \diff_t \varphi \rangle \diff t \right| &\lesssim \int_0^T \| F \|_{L^2(\Omega)} \| \varphi \|_{L^2(\Omega)} \diff t + \int_0^T |u|_{\widetilde{W}^{s+\nu,2}(\Omega)} |\varphi|_{\widetilde{W}^{s-\nu,2}(\Omega)} \diff t
    \\
    &\lesssim \left( \| F \|_{L^2(Q_T)} + \| u \|_{L^2(0,T;\widetilde{W}^{s+\nu,2}(\Omega))} \right) 
      \| \varphi \|_{L^2(0,T;\widetilde{W}^{s-\nu,2}(\Omega))}
  \end{align*}
  where we used the mapping properties of $\calL_A$, which were discussed during the course of the proof of Theorem~\ref{thm:OurResult} and \eqref{eq:HigherIntNLParabolic}. By density, this estimate implies \eqref{eq:BetterTimeDerivNLParabolic}.
\end{proof}

\begin{remark}[nonzero initial data]
  The previous result assumes that $u_0=0$. By linearity we may rewrite the problem as $u = \widetilde{u} + u_0$. Here $\widetilde{u}$ solves \eqref{eq:NLParabolic} with $\widetilde{u}_{|t=0} =0$ and $f -\calL_A u_0$ as right hand side. Smoothness assumptions on $\calL_A u_0$ then can imply a higher integrability and differentiability result.
\end{remark}

\begin{remark}[Lipschitz domains]
  In the case that $\Omega$ is neither convex nor smooth the identification given in Lemma~\ref{lem:FractLaplace} is not possible. The reason for this is that, for $w \in W^{1,2}_0(\Omega)$, $\Delta w \in L^2(\Omega)$ does not imply $w \in W^{2,2}(\Omega) \cap W^{1,2}_0(\Omega)$ anymore. According to \cite[Theorem B.2]{MR1331981} in general Lipschitz domains one may still conclude $w \in W^{3/2-\epsilon,2}(\Omega) \cap W^{1,2}_0(\Omega)$. Thus, one may restrict the value of $\sigma$ in Lemma~\ref{lem:FractLaplace} and, for that restricted range, obtain an analogue of Theorem~\ref{thm:NLParabolic}.
\end{remark}

\subsection{Semi-linear equations}
As in the elliptic case, we can use the result for the linear case to obtain regularity results for a class of semi-linear nonlocal problems 
\begin{equation}
\label{eq:NLSLParabolic}
  \begin{dcases}
    \diff_t u + \calL_A u = F(t, x, u), & x \in \Omega, \ t \in (0,T),
    \\
    u(x,t) = 0, & x \in \R^d \setminus \Omega, \ t \in [0,T), 
    \\
    u(x,0)= u_0(x,0), & x \in \Omega. 
  \end{dcases}
\end{equation}
Assume that the nonlinearity $F : (0,T) \times \Omega \times \R \ni (t, x, z) \mapsto F(t,x,z) \in \R$ satisfies the following conditions:  
\begin{itemize}
  \item[(P1)] (Carath\'eodory) For all $z \in \R$ the mapping $(0,T) \times \Omega \ni (t, x) \mapsto F(t, x, z) \in \R$ is measurable, and $\R \ni z \mapsto F(t, x, u) \in \R$ is continuous for \mae $(t, x) \in (0, T) \times \Omega$.
  
  \item[(P2)] (Lipschitz growth) There exists a constant $\kappa \ge 0$ such that for \mae $(t, x) \in (0,T) \times \Omega$ and all $z, \zeta \in \R$:
  \begin{equation*}
    \left( F(t, x, z) - F(t, x, \zeta) \right)(z - \zeta) \le \kappa |z - \zeta|^2.
  \end{equation*}
  
  \item[(P3)] (Growth) There exists
  \[
    p \in \begin{dcases}
            \left(1 , \frac{4s}{d} \right), & d > 2s,
            \\
            (1,\infty), & d \leq 2s,
          \end{dcases}
  \]
  $g_1 \in L^{p'}(Q_T)$, and $C_2>0$ such that
  \begin{equation*}\label{eq:strict_growth1}
      |F(t, x, z)| \le g_2(t, x) + C_2|z|^p, \qquad \forall z\in \R.
  \end{equation*}
\end{itemize}
Assume, in addition, that the initial condition verifies $u_{0}\in L^{2}(\Omega)$. In this case, the existence of a unique solution $u$ to \eqref{eq:NLSLParabolic}  that belongs to 
\begin{equation*}
 L^\infty(0, T; L^2(\Omega)) \cap L^2(0, T; \widetilde{W}^{s,2}(\Omega)),
\end{equation*}
follows from the pseudo-monotone operator framework presented, for instance, in \cite[Chapter 8]{MR3014456}. Moreover, if $F$ satisfies the restricted growth assumption: 
\begin{equation*}\label{eq:strict_growth2}
    |F(t, x, z)| \le g_2(t, x) + C_2|z|^\rho,
\end{equation*}
where $g_2 \in L^2(0,T; L^2(\Omega))$ and $1 \le \rho \le 1 + \frac{2s}{d}$ (if $d > 2s$), then the composite source term $\widetilde{f}$, defined as,
\begin{equation*}
    \widetilde{f}(t, x)= F(t, x, u(x,t)),
\end{equation*}
is such that $\widetilde{f} \in L^2(0, T; L^2(\Omega))$. Thus, if the initial data is more regular, $u_0 \in \widetilde{W}^{s,2}(\Omega)$, by Proposition \ref{prop:smoothdata} and uniqueness of solutions we can deduce the improved solution regularity:
\begin{equation}
    u \in L^\infty(0, T; \widetilde{W}^{s,2}(\Omega)) \cap C([0,T]; L^2(\Omega)), \qquad \qquad \diff_tu  \in L^2(0, T; L^2(\Omega)).
\end{equation}
If $u_0 = 0, $ we may apply Theorem \ref{thm:NLParabolic}. to deduce the existence of $\delta >0$ so that the solution to the semi-linear problem \eqref{eq:NLSLParabolic} belongs to 
\[
L^2( 0,T; \widetilde{W}^{s+\delta,2}(\Omega)) \cap W^{1/2+\delta}(0,T;L^2(\Omega)).  
\]

\section*{Funding}

TM has been partially supported by NSF grant DMS-2509059. AJS has been partially supported by NSF grant DMS-2409918.

\section*{Disclosure statement}

All the authors declare that they have absolutely zero competing interests.

\section*{Generative AI use}

No generative AI was used to complete any part of this work.

\section*{Data availability statement}

No data was used to produce this article.

\bibliographystyle{amsplain}
\bibliography{./biblio}

@article{ERHARDT20161772,
title = {Higher integrability for solutions to parabolic problems with irregular obstacles and nonstandard growth},
journal = {Journal of Mathematical Analysis and Applications},
volume = {435},
number = {2},
pages = {1772-1803},
year = {2016},
issn = {0022-247X},
doi = {https://doi.org/10.1016/j.jmaa.2015.11.028},
url = {https://www.sciencedirect.com/science/article/pii/S0022247X15010653},
author = {Andr\'e H. Erhardt}
}

@article{Bogelein2011,
  author  = {B\"ogelein, V. and Duzaar, F.},
  title   = {Higher integrability for parabolic systems with non-standard growth and degenerate diffusions},
  journal = {Publicacions Matem\'atiques},
  year    = {2011},
  volume  = {55},
  pages   = {201--250},
  doi     = {10.5565/publmat_55111_10}
}

@article{Kinnunen2000,
  author  = {Kinnunen, Juha and Lewis, John L.},
  title   = {Higher integrability for parabolic systems of p-{L}aplacian type},
  journal = {Duke Mathematical Journal},
  year    = {2000},
  volume  = {102},
  doi     = {10.1215/s0012-7094-00-10223-2}
}

@article{DiNezza2012,
  author  = {Di Nezza, E. and Palatucci, G. and Valdinoci, E.},
  title   = {Hitchhiker's guide to the fractional {S}obolev spaces},
  journal = {Bulletin des Sciences Math{\'e}matiques},
  volume  = {136},
  number  = {5},
  pages   = {521--573},
  year    = {2012}
}

@article{Servadei2012,
  author  = {Servadei, R. and Valdinoci, E.},
  title   = {Mountain pass solutions for non-local elliptic operators},
  journal = {Journal of Mathematical Analysis and Applications},
  volume  = {389},
  number  = {2},
  pages   = {887--898},
  year    = {2012}
}

@article {MR417568,
    AUTHOR = {Meyers, Norman G. and Elcrat, Alan},
     TITLE = {Some results on regularity for solutions of non-linear
              elliptic systems and quasi-regular functions},
   JOURNAL = {Duke Math. J.},
  FJOURNAL = {Duke Mathematical Journal},
    VOLUME = {42},
      YEAR = {1975},
     PAGES = {121--136},
      ISSN = {0012-7094,1547-7398},
   MRCLASS = {35J35},
  MRNUMBER = {417568},
MRREVIEWER = {L.\ A.\ Peletier},
       URL = {http://projecteuclid.org/euclid.dmj/1077310903},
}

@article {MR1607245,
    AUTHOR = {Zhikov, V. V.},
     TITLE = {Meyer-type estimates for solving the nonlinear {S}tokes
              system},
   JOURNAL = {Differ. Uravn.},
  FJOURNAL = {Differentsial\cprime nye Uravneniya},
    VOLUME = {33},
      YEAR = {1997},
    NUMBER = {1},
     PAGES = {107--114, 143},
      ISSN = {0374-0641},
   MRCLASS = {35Q30 (35J60)},
  MRNUMBER = {1607245},
MRREVIEWER = {Alexander\ Yurjevich\ Chebotarev},
}

@article {MR3907738,
    AUTHOR = {Auscher, Pascal and Bortz, Simon and Egert, Moritz and Saari,
              Olli},
     TITLE = {Nonlocal self-improving properties: a functional analytic
              approach},
   JOURNAL = {Tunis. J. Math.},
  FJOURNAL = {Tunisian Journal of Mathematics},
    VOLUME = {1},
      YEAR = {2019},
    NUMBER = {2},
     PAGES = {151--183},
      ISSN = {2576-7658,2576-7666},
   MRCLASS = {35R11 (26A33 35D30 35K90 46B70)},
  MRNUMBER = {3907738},
       DOI = {10.2140/tunis.2019.1.151},
       URL = {https://doi.org/10.2140/tunis.2019.1.151},
}

@article {MR634681,
    AUTHOR = {\v{S}ne\u{\i}berg, I. Ja.}, 
     TITLE = {Spectral properties of linear operators in interpolation
              families of {B}anach spaces},
   JOURNAL = {Mat. Issled.},
  FJOURNAL = {Akademiya Nauk Moldavsko\u i\ SSR. Institut Matematiki s
              Vychislitel\cprime nym Tsentrom. Matematicheskie
              Issledovaniya},
    NUMBER = {2(32)},
    VOLUME = {9},
      YEAR = {1974},
     PAGES = {214--229, 254--255},
      ISSN = {0542-9994},
   MRCLASS = {47A10 (46M35 47B30)},
  MRNUMBER = {634681},
MRREVIEWER = {E.\ Gerlach},
}

@article {MR569253,
    AUTHOR = {Zafran, Misha},
     TITLE = {Spectral theory and interpolation of operators},
   JOURNAL = {J. Functional Analysis},
  FJOURNAL = {Journal of Functional Analysis},
    VOLUME = {36},
      YEAR = {1980},
    NUMBER = {2},
     PAGES = {185--204},
      ISSN = {0022-1236},
   MRCLASS = {47A10 (46M35)},
  MRNUMBER = {569253},
MRREVIEWER = {E.\ Gerlach},
       DOI = {10.1016/0022-1236(80)90099-3},
       URL = {https://doi.org/10.1016/0022-1236(80)90099-3},
}

@article {MR3336922,
    AUTHOR = {Kuusi, Tuomo and Mingione, Giuseppe and Sire, Yannick},
     TITLE = {Nonlocal self-improving properties},
   JOURNAL = {Anal. PDE},
  FJOURNAL = {Analysis \& PDE},
    VOLUME = {8},
      YEAR = {2015},
    NUMBER = {1},
     PAGES = {57--114},
      ISSN = {2157-5045,1948-206X},
   MRCLASS = {35R11 (35B65 35D30 45K05)},
  MRNUMBER = {3336922},
MRREVIEWER = {Pablo\ Ra\'ul\ Stinga},
       DOI = {10.2140/apde.2015.8.57},
       URL = {https://doi.org/10.2140/apde.2015.8.57},
}

@article {Mengesha2025,
    AUTHOR = {Mengesha, Tadele and Salgado, Abner J. and Siktar, Joshua M.},
     TITLE = {Asymptotic compatibility of parametrized optimal design
              problems},
   JOURNAL = {ESAIM Math. Model. Numer. Anal.},
  FJOURNAL = {ESAIM. Mathematical Modelling and Numerical Analysis},
    VOLUME = {59},
      YEAR = {2025},
    NUMBER = {6},
     PAGES = {3069--3105},
      ISSN = {2822-7840,2804-7214},
   MRCLASS = {49M41 (49M25 65R20 74P05)},
  MRNUMBER = {4989257},
MRREVIEWER = {Asatur\ Zh.\ Khurshudyan},
       DOI = {10.1051/m2an/2025084},
       URL = {https://doi.org/10.1051/m2an/2025084},
}

@book {MR482275,
    AUTHOR = {Bergh, J\"oran and L\"ofstr\"om, J\"orgen},
     TITLE = {Interpolation spaces. {A}n introduction},
    SERIES = {Grundlehren der Mathematischen Wissenschaften},
    VOLUME = {No. 223},
 PUBLISHER = {Springer-Verlag, Berlin-New York},
      YEAR = {1976},
     PAGES = {x+207},
   MRCLASS = {46M35},
  MRNUMBER = {482275},
}

@book {MR3753604,
    AUTHOR = {Lunardi, Alessandra},
     TITLE = {Interpolation theory},
    SERIES = {Appunti. Scuola Normale Superiore di Pisa (Nuova Serie)
              [Lecture Notes. Scuola Normale Superiore di Pisa (New
              Series)]},
    VOLUME = {16},
   EDITION = {Third},
 PUBLISHER = {Edizioni della Normale, Pisa},
      YEAR = {2018},
     PAGES = {xiv+199},
      ISBN = {978-88-7642-639-1; 978-88-7642-638-4},
   MRCLASS = {46M35 (46-02 46B70 47D06 47F05)},
  MRNUMBER = {3753604},
       DOI = {10.1007/978-88-7642-638-4},
       URL = {https://doi.org/10.1007/978-88-7642-638-4},
}

@article {MR358326,
    AUTHOR = {Cwikel, Michael},
     TITLE = {On {$(L\sp{po}(A\sb{o}),\,\ L\sp{p\sb{1}}(A\sb{1}))\sb{\theta
              },\,\sb{q}$}},
   JOURNAL = {Proc. Amer. Math. Soc.},
  FJOURNAL = {Proceedings of the American Mathematical Society},
    VOLUME = {44},
      YEAR = {1974},
     PAGES = {286--292},
      ISSN = {0002-9939,1088-6826},
   MRCLASS = {46E30},
  MRNUMBER = {358326},
MRREVIEWER = {Jorgen\ Lofstrom},
       DOI = {10.2307/2040423},
       URL = {https://doi.org/10.2307/2040423},
}

@article {MR159110,
    AUTHOR = {Meyers, Norman G.},
     TITLE = {An {$L\sp{p}$}-estimate for the gradient of solutions of
              second order elliptic divergence equations},
   JOURNAL = {Ann. Scuola Norm. Sup. Pisa Cl. Sci. (3)},
  FJOURNAL = {Annali della Scuola Normale Superiore di Pisa. Classe di
              Scienze. Serie III},
    VOLUME = {17},
      YEAR = {1963},
     PAGES = {189--206},
      ISSN = {0391-173X},
   MRCLASS = {35.42},
  MRNUMBER = {159110},
MRREVIEWER = {Lipman\ Bers},
}

@article {MR1951822,
    AUTHOR = {Triebel, Hans},
     TITLE = {Function spaces in {L}ipschitz domains and on {L}ipschitz
              manifolds. {C}haracteristic functions as pointwise
              multipliers},
   JOURNAL = {Rev. Mat. Complut.},
  FJOURNAL = {Revista Matem\'atica Complutense},
    VOLUME = {15},
      YEAR = {2002},
    NUMBER = {2},
     PAGES = {475--524},
      ISSN = {1139-1138,1988-2807},
   MRCLASS = {46E35 (58D15)},
  MRNUMBER = {1951822},
MRREVIEWER = {Gilles\ Carron},
       DOI = {10.5209/rev\_REMA.2002.v15.n2.16910},
       URL = {https://doi.org/10.5209/rev_REMA.2002.v15.n2.16910},
}

@article {MR430814,
    AUTHOR = {Schmeisser, Hans-J\"urgen and Triebel, Hans},
     TITLE = {Anisotropic spaces. {I}. {I}nterpolation of abstract spaces
              and function spaces},
   JOURNAL = {Math. Nachr.},
  FJOURNAL = {Mathematische Nachrichten},
    VOLUME = {73},
      YEAR = {1976},
     PAGES = {107--123},
      ISSN = {0025-584X,1522-2616},
   MRCLASS = {46M35 (46E35)},
  MRNUMBER = {430814},
MRREVIEWER = {J\"oran\ Bergh},
       DOI = {10.1002/mana.19760730108},
       URL = {https://doi.org/10.1002/mana.19760730108},
}

@book {MR503903,
    AUTHOR = {Triebel, Hans},
     TITLE = {Interpolation theory, function spaces, differential operators},
    SERIES = {North-Holland Mathematical Library},
    VOLUME = {18},
 PUBLISHER = {North-Holland Publishing Co., Amsterdam-New York},
      YEAR = {1978},
     PAGES = {528},
      ISBN = {0-7204-0710-9},
   MRCLASS = {46E35 (35Jxx 46M35)},
  MRNUMBER = {503903},
MRREVIEWER = {Robert\ D.\ Brown},
}

@book {MR3930629,
    AUTHOR = {Amann, Herbert},
     TITLE = {Linear and quasilinear parabolic problems. {V}ol. {II}},
    SERIES = {Monographs in Mathematics},
    VOLUME = {106},
      NOTE = {Function spaces},
 PUBLISHER = {Birkh\"auser/Springer, Cham},
      YEAR = {2019},
     PAGES = {xiv+464},
      ISBN = {978-3-030-11762-7; 978-3-030-11763-4},
   MRCLASS = {46-02 (35Kxx 42B35 46E35 46E40 46F05 46M35 46N20)},
  MRNUMBER = {3930629},
MRREVIEWER = {Jos\'e\ Luis\ Ansorena},
       DOI = {10.1007/978-3-030-11763-4},
       URL = {https://doi.org/10.1007/978-3-030-11763-4},
}

@book {MR1156075,
    AUTHOR = {Dautray, Robert and Lions, Jacques-Louis},
     TITLE = {Mathematical analysis and numerical methods for science and
              technology. {V}ol. 5},
      NOTE = {Evolution problems. I,
              With the collaboration of Michel Artola, Michel Cessenat and
              H\'el\`ene Lanchon,
              Translated from the French by Alan Craig},
 PUBLISHER = {Springer-Verlag, Berlin},
      YEAR = {1992},
     PAGES = {xiv+709},
      ISBN = {3-540-50205-X; 3-540-66101-8},
   MRCLASS = {00A05 (35-01 47-01)},
  MRNUMBER = {1156075},
       DOI = {10.1007/978-3-642-58090-1},
       URL = {https://doi.org/10.1007/978-3-642-58090-1},
}

@article {MR200593,
    AUTHOR = {Kaplan, Stanley},
     TITLE = {Abstract boundary value problems for linear parabolic
              equations},
   JOURNAL = {Ann. Scuola Norm. Sup. Pisa Cl. Sci. (3)},
  FJOURNAL = {Annali della Scuola Normale Superiore di Pisa. Classe di
              Scienze. Serie III},
    VOLUME = {20},
      YEAR = {1966},
     PAGES = {395--419},
      ISSN = {0391-173X},
   MRCLASS = {35.65},
  MRNUMBER = {200593},
MRREVIEWER = {O.\ Hor\'a\v cek},
}

@article {MR4669309,
    AUTHOR = {Baasandorj, Sumiya and Byun, Sun-Sig and Kim, Wontae},
     TITLE = {Self-improving properties of very weak solutions to double
              phase systems},
   JOURNAL = {Trans. Amer. Math. Soc.},
  FJOURNAL = {Transactions of the American Mathematical Society},
    VOLUME = {376},
      YEAR = {2023},
    NUMBER = {12},
     PAGES = {8733--8768},
      ISSN = {0002-9947,1088-6850},
   MRCLASS = {35D30 (35J60 35J70)},
  MRNUMBER = {4669309},
MRREVIEWER = {Seunghyeok\ Kim},
       DOI = {10.1090/tran/9039},
       URL = {https://doi.org/10.1090/tran/9039},
}

@article {MR4507377,
    AUTHOR = {Scott, James M. and Mengesha, Tadele},
     TITLE = {Self-improving inequalities for bounded weak solutions to
              nonlocal double phase equations},
   JOURNAL = {Commun. Pure Appl. Anal.},
  FJOURNAL = {Communications on Pure and Applied Analysis},
    VOLUME = {21},
      YEAR = {2022},
    NUMBER = {1},
     PAGES = {183--212},
      ISSN = {1534-0392,1553-5258},
   MRCLASS = {35B65 (35J70 42B37 45K05)},
  MRNUMBER = {4507377},
       DOI = {10.3934/cpaa.2021174},
       URL = {https://doi.org/10.3934/cpaa.2021174},
}

@article {MR4019094,
    AUTHOR = {Gianazza, Ugo and Schwarzacher, Sebastian},
     TITLE = {Self-improving property of the fast diffusion equation},
   JOURNAL = {J. Funct. Anal.},
  FJOURNAL = {Journal of Functional Analysis},
    VOLUME = {277},
      YEAR = {2019},
    NUMBER = {12},
     PAGES = {108291, 57},
      ISSN = {0022-1236,1096-0783},
   MRCLASS = {35K59 (35B65 35K67)},
  MRNUMBER = {4019094},
MRREVIEWER = {Andrei\ Tarfulea},
       DOI = {10.1016/j.jfa.2019.108291},
       URL = {https://doi.org/10.1016/j.jfa.2019.108291},
}

@article {MR4149740,
    AUTHOR = {Li, Qifan},
     TITLE = {Higher integrability for obstacle problem related to the
              singular porous medium equation},
   JOURNAL = {Bound. Value Probl.},
  FJOURNAL = {Boundary Value Problems},
      YEAR = {2020},
     PAGES = {Paper No. 147, 36},
      ISSN = {1687-2762,1687-2770},
   MRCLASS = {35K59 (35B45 35B65 35K57 35K65)},
  MRNUMBER = {4149740},
       DOI = {10.1186/s13661-020-01445-x},
       URL = {https://doi.org/10.1186/s13661-020-01445-x},
}

@article {MR1282226,
    AUTHOR = {Shi, Peter and Wright, Steve},
     TITLE = {Higher integrability of the gradient in linear elasticity},
   JOURNAL = {Math. Ann.},
  FJOURNAL = {Mathematische Annalen},
    VOLUME = {299},
      YEAR = {1994},
    NUMBER = {3},
     PAGES = {435--448},
      ISSN = {0025-5831,1432-1807},
   MRCLASS = {35B65 (35J45 35Q72 73C02 73C35)},
  MRNUMBER = {1282226},
MRREVIEWER = {Joseph\ J.\ Roseman},
       DOI = {10.1007/BF01459793},
       URL = {https://doi.org/10.1007/BF01459793},
}

@article {Mengesha-Scott-FKorn,
    AUTHOR = {Mengesha, Tadele and Scott, James M.},
     TITLE = {A fractional {K}orn-type inequality for smooth domains and a
              regularity estimate for nonlinear nonlocal systems of
              equations},
   JOURNAL = {Commun. Math. Sci.},
  FJOURNAL = {Communications in Mathematical Sciences},
    VOLUME = {20},
      YEAR = {2022},
    NUMBER = {2},
     PAGES = {405--423},
      ISSN = {1539-6746,1945-0796},
   MRCLASS = {46E35 (35R11 46E40)},
  MRNUMBER = {4374291},
       DOI = {10.4310/CMS.2022.v20.n2.a5},
       URL = {https://doi.org/10.4310/CMS.2022.v20.n2.a5},
}

@book {MR1345385,
    AUTHOR = {Amann, Herbert},
     TITLE = {Linear and quasilinear parabolic problems. {V}ol. {I}},
    SERIES = {Monographs in Mathematics},
    VOLUME = {89},
      NOTE = {Abstract linear theory},
 PUBLISHER = {Birkh\"auser Boston, Inc., Boston, MA},
      YEAR = {1995},
     PAGES = {xxxvi+335},
      ISBN = {3-7643-5114-4},
   MRCLASS = {34Gxx (35Kxx 46M35 46N20 47D06 47N20)},
  MRNUMBER = {1345385},
MRREVIEWER = {Paolo\ Acquistapace},
       DOI = {10.1007/978-3-0348-9221-6},
       URL = {https://doi.org/10.1007/978-3-0348-9221-6},
}

@article {MR216336,
    AUTHOR = {Fujiwara, Daisuke},
     TITLE = {Concrete characterization of the domains of fractional powers
              of some elliptic differential operators of the second order},
   JOURNAL = {Proc. Japan Acad.},
  FJOURNAL = {Proceedings of the Japan Academy},
    VOLUME = {43},
      YEAR = {1967},
     PAGES = {82--86},
      ISSN = {0021-4280},
   MRCLASS = {47.65},
  MRNUMBER = {216336},
MRREVIEWER = {K.\ Gustafson},
       URL = {http://projecteuclid.org/euclid.pja/1195521686},
}

@article {MR3356020,
    AUTHOR = {Bonito, Andrea and Pasciak, Joseph E.},
     TITLE = {Numerical approximation of fractional powers of elliptic
              operators},
   JOURNAL = {Math. Comp.},
  FJOURNAL = {Mathematics of Computation},
    VOLUME = {84},
      YEAR = {2015},
    NUMBER = {295},
     PAGES = {2083--2110},
      ISSN = {0025-5718,1088-6842},
   MRCLASS = {65N30 (65R20)},
  MRNUMBER = {3356020},
MRREVIEWER = {Igor\ Bock},
       DOI = {10.1090/S0025-5718-2015-02937-8},
       URL = {https://doi.org/10.1090/S0025-5718-2015-02937-8},
}

@article {MR1331981,
    AUTHOR = {Jerison, David and Kenig, Carlos E.},
     TITLE = {The inhomogeneous {D}irichlet problem in {L}ipschitz domains},
   JOURNAL = {J. Funct. Anal.},
  FJOURNAL = {Journal of Functional Analysis},
    VOLUME = {130},
      YEAR = {1995},
    NUMBER = {1},
     PAGES = {161--219},
      ISSN = {0022-1236,1096-0783},
   MRCLASS = {35J25 (46E35)},
  MRNUMBER = {1331981},
MRREVIEWER = {H.\ Triebel},
       DOI = {10.1006/jfan.1995.1067},
       URL = {https://doi.org/10.1006/jfan.1995.1067},
}

@book {MR3014456,
    AUTHOR = {Roub\'i{\v c}ek, Tom\'a{\v s}},
     TITLE = {Nonlinear partial differential equations with applications},
    SERIES = {International Series of Numerical Mathematics},
    VOLUME = {153},
   EDITION = {Second},
 PUBLISHER = {Birkh\"auser/Springer Basel AG, Basel},
      YEAR = {2013},
     PAGES = {xx+476},
      ISBN = {978-3-0348-0512-4; 978-3-0348-0513-1},
   MRCLASS = {35-02 (35J60 35K55 35Q30)},
  MRNUMBER = {3014456},
       DOI = {10.1007/978-3-0348-0513-1},
       URL = {https://doi.org/10.1007/978-3-0348-0513-1},
}

\end{document}